\documentclass[10pt]{amsart}

\usepackage{amsfonts}
\usepackage{amssymb}
\usepackage{amsmath, amsthm, latexsym}
\usepackage[all]{xy}
\usepackage{tikz, tikz-cd}
\usepackage{hyperref}
\usepackage[margin=1.2in]{geometry}
\usepackage{color}

\def \C{\mathbb{C}}
\def \Z{\mathbb{Z}}
\def \R{\mathbb{R}}

\def \GL{\operatorname{GL}}

\def \Spec{\operatorname{Spec}}

\def \Sch{\textbf{Sch}}

\def \Hom{\operatorname{Hom}}
\def \k{\textbf{k}}
\def \B{\mathfrak{B}}
\def \PP{\operatorname{PP}}
\def \fr{\textup{fr}}

\newcommand*\abs[1]{\left\lvert#1\right\rvert}

\theoremstyle{plain}
\newtheorem{theorem}{Theorem}[section]
\newtheorem{lemma}[theorem]{Lemma}
\newtheorem{prop}[theorem]{Proposition}
\newtheorem{corollary}[theorem]{Corollary}

\theoremstyle{definition}
\newtheorem{example}[theorem]{Example}
\newtheorem{definition}[theorem]{Definition}
\newtheorem{remark}[theorem]{Remark}

\begin{document}

\title{Moduli of toric principal bundles}

\author{Shaoyu Huang}
\address{Department of Mathematics, University of Pittsburgh,
Pittsburgh, PA, USA.}
\email{shh123@pitt.edu}

\author{Kiumars Kaveh}
\address{Department of Mathematics, University of Pittsburgh, Pittsburgh, PA, USA.}
\email{kaveh@pitt.edu}

\thanks{The second author is partially supported by a
Simons Foundation Collaboration Grants for Mathematicians (Grant ID: 210099) and a National Science Foundation Grant 
(Grant ID: 1200581).}

\date{\today}

\begin{abstract}
Let $G$ be a reductive algebraic group.
A toric principal $G$-bundle is a principal $G$-bundle over a toric variety together with a torus action commuting with the $G$-action. In \cite{Kaveh-Manon}, extending the Klyachko classification of toric vector bundles, toric principal bundles are classified by \emph{piecewise linear maps} to the (extended) Tits building of $G$. In this paper, we use the classification in \cite{Kaveh-Manon} to construct a moduli space of (framed) toric principal bundles with given total equivariant characteristic class, as a locally closed subvariety of a product of partial flag varieties. This extends the construction of moduli of toric vector bundles by Sam Payne \cite{Sam}.
\end{abstract}

\maketitle
\tableofcontents

\noindent{\bf This is a preliminary version. Comments are welcomed.}
\bigskip

\section{Introduction}
Throughout $\k$ denotes the base field which we take to be algebraically closed. We let $T\cong (\k^*)^n$ be an $n$-dimensional algebraic torus over $\k$. We let $M$ and $N$ be the lattices of characters
and cocharacters of $T$ respectively. Let $N_{\R}=N\otimes_{\Z}\R$ and $M_{\R}=M\otimes_{\Z}\R$. Let $\Sigma$ be a fan in $N_{\R}\cong {\R}^n$ and $X_\Sigma$ be its associated toric variety.

Let $G$ be a connected reductive algebraic group over $\k$. A \emph{toric principal $G$-bundle} is a principal $G$-bundle over a toric variety $X_\Sigma$ together with a torus action that lifts that of $X_\Sigma$ and commutes with the $G$-action (\cite{Biswas}).

Fix a point $x_0$ in the open torus orbit in $X_\Sigma$. By a \emph{framed} toric principal bundle, we mean a toric principal bundle $\mathcal{P}$ together with a choice of $p_0\in\mathcal{P}_{x_0}$, where $\mathcal{P}_{x_0}$ is the fiber over the distinguished point $x_0$.

In this paper, we construct a moduli space $\mathcal{M}^{\fr}_{G, \Psi}$ of framed toric principal $G$-bundles on $X_\Sigma$ with fixed total equivariant characteristic class $\Psi$. We construct $\mathcal{M}^{\fr}_{G, \Psi}$ as a locally closed subset of a product of partial flag varieties of $G$. We show that $\mathcal{M}^{\fr}_{G, \Psi}$ is a coarse moduli space, but in fact we believe that it is moreover a fine moduli space (Section \ref{sec-fine-moduli}).

\begin{remark}
In \cite{Payne-moduli}, Payne constructs the moduli space of (framed) toric vector bundles of rank $r$ with fixed equivariant Chern classes, as a locally closed subvariety of a product of partial flag varieties. Our construction of $\mathcal{M}^{\fr}_{G, \Psi}$ is an extension of Payne's construction. It seems that our equalities and inequalities cutting out the moduli space in the product of partial flag varieties (in the case $G=\GL(r)$) are different from those in \cite{Payne-moduli} (called ``rank conditions").     
\end{remark}

\begin{remark}
If one fixes an embedding of $G$ into a general linear group $\GL(N)$, one should be able to realize the moduli space $\mathcal{M}^{\fr}_{G, \Psi}$ as a subvariety of the moduli space of rank $N$ toric vector bundles with fixed equivariant Chern classes. We would like to point out that our construction of the moduli space $\mathcal{M}^{\fr}_{G, \Psi}$ and its embedding in a product of partial flag varieties of $G$ in this paper, is intrinsic to $G$ and does not involve a choice of embedding of $G$ into a $\GL(N)$.      
\end{remark}

In \cite{Kaveh-Manon}, framed toric principal bundles are classified by piecewise linear maps $\Phi:\abs{\Sigma}\to\tilde{\mathfrak{B}}(G)$. Here $\abs{\Sigma}$ is the support of the fan $\Sigma$, namely, the union of the cones in $\Sigma$ and  $\tilde{\mathfrak{B}}(G)$ denotes the \emph{extended Tits building of $G$} (see Definition \ref{def-ext-bldg}). Roughly speaking, $\tilde{\mathfrak{B}}(G)$ can be realized as the union 
\begin{equation} \label{eq-ext-bulid}
    \tilde{\mathfrak{B}}(G)=\bigcup_{H\subset G}\Lambda^\vee_{\R}(H),
\end{equation}
where the union is over all the maximal tori $H$ in $G$ and $\Lambda^\vee(H)$ denotes the cocharacter lattice of $H$ and $\Lambda^\vee_{\R}(H)=\Lambda^\vee(H)\otimes_{\Z}\R$. In \eqref{eq-ext-bulid}, the $\R$-vector spaces $\Lambda^\vee_{\R}(H)$ overlap each other according to how apartments in the Tits building of $G$ overlap. We call $\tilde{A}(H):=\Lambda^\vee_{\R}(H)$ the \emph{extended apartment} associated to $H$.

Throughout the paper, we fix a maximal torus $H \subset G$ and let $S(H)^W_\R$ denote the $\R$-algebra of $W$-invariant polynomials on the cocharacter lattice of $H$ where $W$ is the Weyl group of $(G, H)$. 

Let $\mathcal{P}$ be a toric principal $G$-bundle. The algebraic Chern-Weil homomorphism, corresponding to the principal bundle $\mathcal{P}$, is an $\R$-algebra homomorphism $S(H)^W_\R \to A^*(X_\Sigma; \R)$. It assigns a characteristic class in $A^*(X_\Sigma;\R)$ to a $W$-invariant polynomial $q$. 
Similarly, the equivariant Chern-Weil homomorphism is an $\R$-algebra homomorphism $\Psi = \Psi_\mathcal{P}$ which assigns to $q \in S(H)^W_\R$ an equivariant characteristic class $\Psi(q) \in A^*_T(X_\Sigma;\R)$. We recall that the equivariant Chow cohomology ring $A^*_T(X_\Sigma;\R)$ is naturally isomorphic to the algebra of piecewise polynomial functions $\PP^*(\Sigma)$ on $N_{\R}$ with respect to $\Sigma$ (see \cite{Brion, Sam}). We call the homomorphism $\Psi=\Psi_\mathcal{P}: S(H)^W_\R \to \PP^*(\Sigma)$, the \emph{total equivariant characteristic class} of $\mathcal{P}$.

In \cite{Kaveh-Manon}, it is observed that any $W$-invariant polynomial $q$ naturally gives rise to a map $\tilde{q}: \tilde{\mathfrak{B}}(G) \to \R$ and if $(\mathcal{P}, p_0)$ is a framed toric principal $G$-bundle with piecewise linear map $\Phi:\abs{\Sigma}\to\tilde{\mathfrak{B}}(G)$,  then the equivariant characteristic class $\Psi(q)$ of $\mathcal{P}$ corresponding to $q$ is given by $\tilde{q}\circ \Phi$ (see \cite[Section 3]{Kaveh-Manon}).

Let $\Psi: S(H)^W_\R \to \PP^*(\Sigma)$ be an $\R$-algebra homomorphism. We construct a coarse moduli space $\mathcal{M}^{\fr}_{G,\Psi}$ of all framed toric principal $G$-bundles $(\mathcal{P}, p_0)$ on $X_\Sigma$ whose total equivariant characteristic class is $\Psi$. That is, for any $\mathcal{P} \in \mathcal{M}^{\fr}_{G, \Psi}$ and any $q\in S(H)^W_\R$, the equivariant characteristic class of $\mathcal{P}$ corresponding to $q$ coincides with $\Psi(q)$.  

Let $\Phi:\abs{\Sigma}\to\tilde{\mathfrak{B}}(G)$ be the piecewise linear map corresponding to $(\mathcal{P}, p_0)$. Let $\rho\in\Sigma(1)$ be a ray with primitive vector $v_\rho$. Then $\Phi(v_\rho) \in \tilde{\mathfrak{B}}(G)$ determines a unique parabolic subgroup $P_\rho$ of $G$. It is the parabolic subgroup corresponding to the smallest simplex in the Tits building of $G$ containing $\Phi(v_\rho)$. One notes that, for each ray $\rho$, the value $\Phi(v_\rho)$ is determined by $\Psi$ and the parabolic subgroup $P_\rho$. It follows that the piecewise linear map $\Phi$, and hence the framed toric principal bundle $(\mathcal{P}, p_0)$, is uniquely determined by $\Psi$ and the collection of parabolic subgroups $\{P_\rho\mid \rho\in\Sigma(1)\}$.

Throughout we also fix a Borel subgroup $B$ in $G$ containing the maximal torus $H$. This is equivalent to the choice of a positive Weyl chamber in the cocharacter lattice of $H$. We will refer to a parabolic subgroup containing $B$ as a standard parabolic subgroup. We use the choice of $B$ to enumerate the conjugacy classes of parabolic subgroups in $G$ by standard parabolic subgroups. 

{We note that not every homomorphism $\Psi: S(H)_\R^W \to \PP^*(\Sigma)$ is in the image of equivariant Chern-Weil homomorphism. In other words, $\Psi$ may not be the equivariant total characteristic class of any toric principal bundle over $X_\Sigma$ in which case the moduli space is empty. When $\Psi$ is indeed the total equivariant characteristic class of a toric principal bundle then it determines a collection of standard parabolic subgroups $\{Q_\rho \mid \rho \in \Sigma(1)\}$ as follows:}
Let $\rho \in \Sigma(1)$ be a ray. Evaluating $\Psi$ at the primitive vector $v_\rho$ we get a homomorphism $\Psi_\rho: S(H)_\R^W \to \R$ which in turn determines a $W$-orbit in the $\R$-vector space $\Lambda^\vee(H)_\R = \Lambda^\vee \otimes \R$, where $\Lambda^\vee(H)$ denotes the cocharacter lattice of $H$. We let $Q_\rho$ be the standard parabolic subgroup corresponding to this $W$-orbit (in other words, the standard parabolic subgroup corresponding to the minimal face of the positive Weyl chamber intersecting this orbit).

\begin{theorem}  \label{th-intro-main}
    Let $\Sigma$ be a complete fan. Let $\Psi: S(H)^W_\R \to \PP^*(\Sigma)$ an $\R$-algebra homomorphism with corresponding standard parabolic subgroups $\{Q_\rho \mid \rho \in \Sigma(1)\}$. Then the set $\mathcal{M}_{G,\Psi}^{\fr}$ of isomorphism classes of framed toric principal $G$-bundles on $X_\Sigma$ with total equivariant characteristic class $\Psi$ can naturally be identified with a locally closed subset of the variety $\prod_{\rho\in\Sigma(1)}G/ Q_\rho$. {More precisely, $\mathcal{M}_{G,\Psi}^{\fr}$ is a coarse moduli space}.
\end{theorem}

The group $G$ acts on $\prod_{\rho\in\Sigma(1)}G/ Q_\rho$ by acting on each $G/Q_\rho$ from left. Moreover, this action leaves $\mathcal{M}_{G,\Psi}^{\fr}$ invariant. The following is immediate from Theorem \ref{th-intro-main}.
\begin{corollary}
The quotient stack $\mathcal{M}_{G,\Psi}^{\fr} / G$ is a {(coarse)} moduli stack of toric principal $G$-bundles on $X_\Sigma$ with total equivariant characteristic class $\Psi$.  
\end{corollary}

\begin{remark} \label{rem-fine-moduli-space}
{In \cite{Payne-moduli} a fine moduli space is constructed for toric vector bundles with given total equivariant Chern class. We believe that $\mathcal{M}_{G,\Psi}^{\fr}$ is in fact a \emph{fine} moduli space as well. That is, the natural principal $G$-bundle on it is a universal bundle classifying all families of framed toric principal $G$-bundles on $X_\Sigma$. In Section \ref{sec-fine-moduli} we sketch some ideas towards a proof of this.}      
\end{remark}

\bigskip
\noindent{\bf Acknowledgement:} The second author is supported by the National Science Foundation grant (DMS-2101843). We thank Mainak Poddar, Christopher Manon, Sam Payne and Roman Fedorov for useful discussions. 

\section{Background on Tits Buildings}
\subsection{Tits building of a linear algebraic group}\label{subsec-building}
{In this section we review some background material about the Tits building associated to a linear algebraic group following \cite{Kaveh-Manon, Huang-Kaveh}}. 

A \emph{building} is a pair $(\Delta,\mathcal{A})$ where $\Delta$ is a simplicial complex and $\mathcal{A}$ is a family of subcomplexes $A$ (called \emph{apartments}) satisfying certain axioms.

Let $G$ be a linear algebraic group over a field $\textbf{k}$. To $G$ there corresponds a certain kind of simplicial complex called the \emph{Tits building} of $G$ which we denote by $\Delta(G)$. A parabolic subgroup $P$ of $G$ is \emph{parabolic} if $P$ contains a conjugate of $B$. The set of simplices in $\Delta(G)$ is the set of parabolic subgroups of $G$ ordered by reverse inclusion. The apartments in $\Delta(G)$ correspond to maximal tori in $G$. For a maximal torus $H \subset G$, the corresponding apartment consists of parabolic subgroups containing $H$. Borel subgroups correspond to the maximal simplices, called \emph{chambers}, in $\Delta(G)$. Since every parabolic subgroup of $G$ contains the solvable radical $R(G)$, the simplicial complexes $\Delta(G)$ and $\Delta(G/R(G))$ are isomorphic. 

\begin{example}
Consider $G = \GL(r)$. A parabolic subgroup $P$ of $\GL(r)$ is the stabilizer of a flag $F_\bullet=(\{0\}\subset F_1\subset\cdots\subset F_k=\C^r)$. This gives a one-to-one correspondence between the simplices in the Tits building of $\GL(r)$ and flags in $\C^r$. In particular, Borel subgroups are stabilizers of complete flags.
\end{example}

In fact, there is a topological space $\B(G)$ together with a triangulation in which simplices in the triangulation (which are subsets of $\B(G)$ homeomorphic to standard simplices) are in one-to-one correspondence with the simplices in $\Delta(G)$ and intersect according to how simplices in $\Delta(G)$ intersect. When $G$ is semisimple, $\B(G)$ can be constructed as follows. For each maximal torus $H \subset G$ let $\Lambda^\vee(H)$ be its cocharacter lattice and let $\Lambda^\vee_{\R}(H)= \Lambda^\vee(H) \otimes_{\Z}\R$. The apartment corresponding to $H$ is $A(H)$, the triangulation of the unit sphere in $\Lambda^\vee_{\R}(H)$ obtained by intersecting with the Weyl chambers. Two simplices, in different apartments, are glued together if the corresponding faces represent the same parabolic subgroup in $G$. 

\begin{definition}
\label{def-B(G)}
We can glue the Coxeter complexes $A(H)$ in the $\Lambda^\vee_{\R}(H)$, for all maximal tori $H$, along their common simplices. The resulting topological space is $\B(G)$  
\end{definition}

In our notation, we distinguish between the building as an abstract simplicial complex, i.e. $\Delta(G)$, and as a topological space, i.e. $\B(G)$. By abuse of terminology we refer to both $\Delta(G)$ and $\B(G)$ as the Tits building of $G$. 



For a linear algebraic group $G$, let $G_{\textup{ss}} = G / R(G)$ be its semisimple quotient. Similarly to the previous construction in the semisimple case, we define $\tilde{\B}(G)$ to be the topological space obtained by gluing the vector spaces $\Lambda^\vee_\R(H)$ for all maximal tori $H \subset G$, along their common faces of Weyl chambers. When $G$ is reductive, the topological space $\tilde{\B}(G)$ is the Cartesian product of $\tilde{\B}(G_{\textup{ss}})$ with the real vector space $\Lambda^\vee(Z) \otimes_\Z \R$, where $Z=Z(G)^\circ$ is the connected component of the identity in the center of $G$. 

\begin{definition}[Extended Tits building of a linear algebraic group]      \label{def-ext-bldg}
For a linear algebraic group $G$, we refer to $\tilde{\B}(G)$ (above) as the \emph{extended Tits building} of $G$. Also, for a maximal torus $H$, we refer to $\Lambda^\vee_\R(H)$ as the \textit{extended apartment} of $H$ and denote it by $\tilde{A}(H)$. When $G$ is semisimple, the extended Tits building $\tilde{\B}(G)$ is the cone over the Tits building of $G$.

We denote by $\tilde{\B}_\Z(G)$ the subset of $\tilde{\B}(G)$ obtained by gluing the lattices $\Lambda^\vee(H)$, for all maximal tori $H$,
and call it the set of \emph{lattice points in the extended Tits building of $G$}.
\end{definition}

\begin{remark}
Our terminology of an \emph{extended Tits building} is motivated by a similar term, namely  \emph{an extended Bruhat-Tits building}, from the theory of Bruhat-Tits buildings for algebraic groups over valued fields (see \cite{Tits} as well as \cite[Remark 1.23]{RTW}).
\end{remark}




\subsection{One-parameter subgroups and Tits building} 
\label{subsec-1-para-subgp}
In this section, following \cite[Section 1.3]{Kaveh-Manon}, we recall a natural way to realize the set of lattice points in $\tilde{\mathfrak{B}}(G)$ as certain equivalence classes of one-parameter subgroups in $G$ (Proposition \ref{prop-tilde-B-one-para}).  
See \cite[Section 1.3]{Kaveh-Manon} for details and proofs. This construction of the Tits building from one-parameter subgroups also appears, in somewhat different form, in \cite[Section 2.2]{Mumford}.

Let $G$ be a linear algebraic group. A \emph{one-parameter subgroup} of $G$ is a homomorphism of algebraic groups $\lambda:\mathbb{G}_m\to G$. We say two one-parameter subgroups $\lambda_1$ and $\lambda_2$ are \emph{equivalent}  if $\lim_{s\to0}\lambda_1(s)\lambda_2(s)^{-1}$ exists in $G$ and write it as $\lambda_1\sim\lambda_2$. One shows that this is indeed an equivalence relation. 

\begin{definition}[Parabolic subgroup associated to a one-parameter subgroup]
For a one-parameter subgroup $\lambda:\mathbb{G}_m\to G$, define
    \begin{equation*}
        P_\lambda = \{g\in G\mid \lim_{s\to 0}\lambda(s)g\lambda(s)^{-1}\text{ exists in } G\}.
    \end{equation*}
One shows that $P_\lambda$ is a parabolic subgroup in $G$. Equivalently, we can define $P_\lambda$ using the equivalence relation $\sim$ (see \cite[Proposition 1.8]{Kaveh-Manon}): 
\begin{equation}  \label{equ-P-lambda}
P_\lambda = \{ g \in G \mid g \lambda g^{-1} \sim \lambda\}.
\end{equation}
\end{definition}
It is straightforward to check that equivalent one-parameter subgroups give the same corresponding parabolic subgroups. One also shows that, for a maximal torus $H \subset G$, no two one-parameter subgroups in $\Lambda^\vee(H)$ are equivalent. Moreover, if a one-parameter subgroup $\lambda \in \Lambda^\vee(H)$ lies in the relative interior of a face of a Weyl chamber, the parabolic subgroup $P_\lambda$ is exactly the parabolic subgroup corresponding to this face. Putting these facts together, we have the following (\cite[Corollary 1.11]{Kaveh-Manon}).
\begin{prop}  \label{prop-tilde-B-one-para}
The set $\tilde{\mathfrak{B}}_\Z(G)$ can naturally be identified with the set of equivalence classes of one-parameter subgroups of $G$.
\end{prop}



A homomorphism of linear algebraic groups induces a map between the corresponding extended Tits buildings. The above realization of the extended Tits building in terms of equivalence classes of one-parameter subgroups gives an easy way to construct this map.  
Let $\alpha: G\to G'$ be a homomorphism of linear algebraic groups. If $\lambda:\mathbb{G}_m\to G$ is a one-parameter subgroup of $G$, then $\alpha\circ\lambda$ is a one-parameter subgroup of $G'$. The map $\lambda \mapsto \alpha\circ\lambda$ respects the equivalence classes and thus gives a well-defined map $\hat{\alpha}: \tilde{\mathfrak{B}}_\Z(G)\to\tilde{\mathfrak{B}}_\Z(G')$. This extends to a map $\hat{\alpha}: \tilde{\B}(G) \to \tilde{\B}(G')$.
Because the image of a torus in $G$ is a torus in $G'$ and every torus lies in a maximal torus the map $\hat{\alpha}$ sends an extended apartment for $G$ to an extended apartment for $G'$.  

\section{Background on toric principal bundles}
In this section we review the classification of (framed) toric principal bundles in \cite{Kaveh-Manon}. Let $T \cong \mathbb{G}_m^n$ denote an $n$-dimensional algebraic torus over an algebraically closed field $\k$. We let $M$ and $N$ denote its character and cocharacter lattices respectively. We also denote by $M_\R$ and $N_\R$ the $\R$-vector spaces spanned by $M$ and $N$. 
Let $\Sigma$ be a (finite rational polyhedral) fan in $N_\R$ and let $X_\Sigma$ be the corresponding toric variety. Also $U_\sigma$ denotes the invariant affine open subset in $X_\Sigma$ corresponding to a cone $\sigma \in \Sigma$. We denote the support of $\Sigma$, that is the union of all the cones in $\Sigma$, by $|\Sigma|$. For each $i$, $\Sigma(i)$ denotes the subset of $i$-dimensional cones in $\Sigma$. In particular, $\Sigma(1)$ is the set of rays in $\Sigma$. For each ray $\rho \in \Sigma(1)$ we let $v_\rho$ be the shortest non-zero integral vector along $\rho$, i.e. $v_\rho$ is the shortest non-zero integral vector on $\rho$.

\subsection{Classification of framed toric principal bundles}
\label{subesc-classification-tpbs}
Throughout we fix a point $x_0$ in the open torus orbit in $X_\Sigma$. It gives an identification of the torus $T$ with the open orbit via $t \mapsto t \cdot x_0$. 

We start by recalling the notion of a principal bundle. Let $G$ be an algebraic group, a \emph{principal $G$-bundle over a variety $X$} is a fiber bundle $\mathcal{P}$ over $X$ with an action of $G$ such that $G$ preserves each fiber and the action is free and transitive. Throughout, we take the action of $G$ on $\mathcal{P}$ to be a \emph{right} action. 

Let $G, G'$ be algebraic groups and $\mathcal{P}$ (respectively $\mathcal{P}'$) be a principal $G$-bundle (respectively $G'$-bundle) over $X$. A \emph{morphism of principal bundles with respect to a homomorphism of algebraic groups $\alpha:G\to G'$} is a bundle map $F:\mathcal{P}\to\mathcal{P}'$
 such that
    \begin{equation*}
        F(z\cdot g)=F(z)\cdot\alpha(g), \ \forall z\in\mathcal{P}, \forall g\in G.
    \end{equation*}
We refer to a morphism between toric principal $G$-bundles, with respect to the identity homomorphism $G \to G$, simply as a \emph{morphism of toric principal $G$-bundles}.

\begin{definition}[Toric principal bundle]   \label{def-tpb}
    Let $X_\Sigma$ be the toric variety associated to a fan $\Sigma$ and $G$ an algebraic group. A \emph{toric principal $G$-bundle over $X_\Sigma$} is a principal $G$-bundle $\mathcal{P}$ together with a torus action lifting that of $X_\Sigma$,
such that the $T$-action and the $G$-action on $\mathcal{P}$ commute. More precisely, 
$\forall t\in T, \forall x\in X_\Sigma, \forall z\in\mathcal{P}_x$ we have:
    \begin{align*}
        t:\mathcal{P}_x&\to\mathcal{P}_{t\cdot x}, \\
        t\cdot (z\cdot g) &= (t\cdot z)\cdot g.
    \end{align*}
    
Recall that we have fixed a point $x_0$ in the open torus orbit in $X_\Sigma$. We call a toric principal $G$-bundle $\mathcal{P}$ together with a choice of a point $p_0\in\mathcal{P}_{x_0}$ a \emph{framed toric principal $G$-bundle}.
\end{definition}

\begin{definition}   \label{def-morphism-tpb}
A \emph{morphism of toric principal bundles}  is a morphism $F$ of principal bundles (with respect to some homomorphism $\alpha$ as above) that is also $T$-equivariant.
    A \emph{morphism of framed principal bundles} $(\mathcal{P}, p_0) \to (\mathcal{P}', p_0')$ is a morphism $F$ that sends $p_0\in\mathcal{P}_{x_0}$ to
    $p_0'\in\mathcal{P}'_{x_0}$.
\end{definition}

The following is the main combinatorial gadget to classify (framed) toric principal bundles. It can be thought of as a generalization of a real-valued piecewise linear function $\varphi: |\Sigma| \to \R$. 
\begin{definition}[Piecewise linear map] \label{def-pwl}
Let $G$ be a linear algebraic group with $\tilde{\B}(G)$, the extended Tits building of $G$. Let $\Sigma$ be a fan in $N_\R$, we say that a map $\Phi:\abs{\Sigma}\to\tilde{\mathfrak{B}}(G)$ is a \emph{piecewise linear map} if:
    \begin{enumerate}
        \item[(a)] For each cone $\sigma\in\Sigma$, $\Phi(\sigma)$ lies in an extended apartment $\tilde{A}_\sigma=\Lambda_\R^\vee(H_\sigma)$.
        \item[(b)] For each cone $\sigma\in\Sigma$, the restriction $\Phi|_\sigma:\sigma\to\tilde{A}_\sigma$ is an $\R$-linear map.
    \end{enumerate}
We say that a piecewise linear map $\Phi$ is \emph{integral} if $\Phi$ sends lattice points to lattice points, i.e. for any $\sigma \in \Sigma$, $\Phi(\sigma \cap N) \subset \Lambda^\vee(H_\sigma)$.
\end{definition}

\begin{definition}[Equivariant triviality]
 We say that a toric principal bundle $\mathcal{P}$ on an affine toric variety $U_\sigma$ is \emph{equivariantly trivial} if there exists a toric principal $G$-bundle isomorphism between $\mathcal{P}$ and $U_\sigma\times G$, where $T$ acts on $U_\sigma \times G$ via an algebraic group homomorphism $\phi_\sigma: T\to G$ by:
\begin{equation*}
        t\cdot(x, g) = (t\cdot x, \phi_\sigma(t)g),\ \forall t\in T, \forall x\in U_\sigma, \forall g\in G.
\end{equation*}
\end{definition}

\begin{definition}[Local equivariant triviality]
Let $\mathcal{P}$ be a toric principal $G$-bundle on a toric variety $X_\Sigma$. We say that $\mathcal{P}$ is \emph{locally equivariantly trivial} if the restriction $\mathcal{P}|_{U_\sigma}$ is equivariantly trivial for any cone $\sigma\in\Sigma$.
\end{definition}

The following gives a classification of locally equivariantly trivial framed toric principal bundles in terms of piecewise linear maps (\cite[Theorem 2.4]{Kaveh-Manon}).

\begin{theorem} \label{th-KM}
Let $G$ be a linear algebraic group over $\mathbf{k}$. 
\begin{itemize}
\item[(a)] There is a one-to-one correspondence between the isomorphism classes of locally equivariantly trivial framed toric principal $G$-bundles $\mathcal{P}$ over $X_\Sigma$ and the integral piecewise linear maps $\Phi:\abs{\Sigma}\to\tilde{\mathfrak{B}}(G)$.
    
\item[(b)] Moreover, let $\alpha:G\to G'$ be a homomorphism of linear algebraic groups. Let $(\mathcal{P}, p_0)$ (respectively $(\mathcal{P}', p'_0)$) be a locally equivariantly trivial framed toric principal $G$-bundle (respectively $G'$-bundle) with corresponding 
    piecewise linear map $\Phi:\abs{\Sigma}\to\tilde{\mathfrak{B}}(G)$ (respectively $\Phi':\abs{\Sigma}\to\tilde{\mathfrak{B}}(G')$). Then there is a morphism of framed toric principal bundles $F:\mathcal{P}\to\mathcal{P}'$ with respect to $\alpha$ if and only if $\Phi' = \hat{\alpha}\circ\Phi$.
\end{itemize}
\end{theorem}

It is shown in \cite[Theorem 4.1]{Dey} that if $G$ is reductive then any toric principal $G$-bundle is locally equivariantly trivial. Thus Theorem \ref{th-KM} immediately implies the following.

\begin{corollary} \label{cor-KM} Let $G$ be a reductive algebraic group over $\mathbf{k}$. 
\begin{itemize}
\item[(a)] There is a one-to-one correspondence between the isomorphism classes of framed toric principal $G$-bundles $\mathcal{P}$ over $X_\Sigma$ and the integral piecewise linear maps $\Phi:\abs{\Sigma}\to\tilde{\mathfrak{B}}(G)$.
    
\item[(b)] Moreover, let $\alpha:G\to G'$ be a homomorphism of reductive algebraic groups. Let $\mathcal{P}$ (respectively $\mathcal{P}'$) be a framed toric principal $G$-bundle (respectively $G'$-bundle) with corresponding 
    piecewise linear map $\Phi:\abs{\Sigma}\to\tilde{\mathfrak{B}}(G)$ (respectively $\Phi':\abs{\Sigma}\to\tilde{\mathfrak{B}}(G')$). Then there is a morphism of framed toric principal bundles $F:\mathcal{P}\to\mathcal{P}'$ with respect to $\alpha$ if and only if $\Phi' = \hat{\alpha}\circ\Phi$.
\end{itemize}
\end{corollary}


The following is a simple corollary of Theorem \ref{th-KM}(b).
\begin{lemma}   \label{lem-KM}
    Let $(\mathcal{P},p_0)$ be a locally equivariantly trivial framed toric principal $G$-bundle with corresponding integral piecewise linear map $\Phi:\abs{\Sigma}\to\tilde{\mathfrak{B}}(G)$. Then for any $g_0\in G$, the corresponding
integral piecewise linear map for the framed toric principal $G$-bundle $(\mathcal{P},p_0\cdot g_0)$ is $\hat{\alpha}_{g_0}\circ \Phi$, where $\alpha_{g_0}: G \to G$ is the conjugation homomorphism $x \mapsto g_0^{-1} x g_0$.
\end{lemma}
\begin{proof}
The right action by $g_0$ gives a morphism of framed toric principal bundles from $(\mathcal{P}, p_0)$ to $(\mathcal{P}, p_0 \cdot g_0)$ with respect to the conjugation homomorphism $\alpha_{g_0}: G \to G$. Theorem \ref{th-KM}(b) then implies that the piecewise linear map of $(\mathcal{P}, p_0 \cdot g_0)$ is $\hat{\alpha}_{g_0}\circ \Phi$.
\end{proof}

\subsection{Equivariant characteristic classes}
We assume $G$ is a reductive algebraic group. Let $\mathcal{P}$ be a toric principal $G$-bundle over $X_\Sigma$ with torus $T$. Let $BT, BG$ be classifying spaces of $T$ and $G$ with $ET$ and $EG$ the corresponding universal bundles respectively.
\begin{definition}
    If $X$ and $Y$ are left $G$-spaces, and we let $G$ acting on $X\times Y$ on left by $g\cdot(x, y)=(g\cdot x, g\cdot y)$ for $g\in G$, then $X\times_G Y$ is the orbit space $(X\times Y)/G$. 
\end{definition}

Consider $\mathcal{P}_T:=ET\times_T \mathcal{P}$ and $X_T:=ET\times_T X_\Sigma$. The equivariant cohomology $H^*_T(X_\Sigma;\C)$ is the usual cohomology $H^*(X_T;\C)$.
Since $\mathcal{P}_T$ is a principal $G$-bundle over $X_T$, by the universal property, there exists a continuous map $f:X_T\to BG$. It induces a homomorphism $f^*:H^*(BG;\C)\to H^*(X_T;\C)=H_T^*(X_\Sigma;\C)$.
We have $H^*(BG;\C)\simeq\C[\mathfrak{g}]^G$, where $\C[\mathfrak{g}]^G$ consists of polynomial functions on ${\mathfrak{g}}$ which are invariant under the adjoint action of $G$ and $\mathfrak{g}$ is the Lie algebra of $G$.
Let $\mathfrak{h}$ be a Cartan subalgebra of $\mathfrak{g}$ and $W$ the associated Weyl group. The inclusion $\mathfrak{h}\subset\mathfrak{g}$ induces an isomorphism
$\C[\mathfrak{h}]^W\simeq\C[\mathfrak{g}]^G$ by the Chevalley restriction theorem. The equivariant Chern-Weil homomorphism is $f^*:\C[\mathfrak{h}]^W\to H^*_T(X_\Sigma;\C).$

However, $EG$ does not exist as a finite dimensional algebraic variety in general. This can be solved by taking an open subset $U_m$ of a representation of $G$ such that $G$ acts freely on $U_m$ and the codimension $m$ of the complement of $U$ is sufficiently large (\cite[Theorem 1.1]{Totaro}).
Following \cite{Edidin}, the $i$-th $G$-equivariant Chow group of $X$ is defined to be $A^i_G(X)=A^i(X\times_G U_m)$, where $i\le m$.

Fix a maximal torus $H\subset G$ and let $W$ be the Weyl group of $(G,H)$. Let $S(H)_\R$ be the $\R$-algebra generated by the character lattice of $H$. \cite[Theorem 1]{Edidin2} shows that the $G$-equivariant Chow cohomology ring $A_G^*(\text{pt};\R)$ of a point is naturally isomorphic to the $W$-invariant $\R$-algebra $S(H)^W_\R$.

From \cite[Lemma 1.6]{Totaro}, there exists an affine space bundle $\pi:X'\to X_\Sigma$ and a map $f:X'\to B_mG$ such that the pullbacks $\pi^*\mathcal{P}$ and $f^* U_m$ are isomorphic, where $B_mG = U_m/ G$. Therefore, the Chern-Weil homomorphism, for $i<m$, is $S(H)^W_\R\to A^i(X';\R)\simeq A^i(X;\R)$.

For a complete toric variety $X_\Sigma$, the equivariant Chow cohomology ring $A_T^*(X_\Sigma;\R)$ is naturally isomorphic to the algebra of piecewise polynomial functions on $N_\R$ with respect to $\Sigma$ (\cite{Sam}). The isomorphism is given by the localization map:
$$A_T^*(X_\Sigma;\R)\to\bigoplus_{\sigma\in\Sigma(n)}A_T^*(\{x_\sigma\};\R)\simeq\bigoplus_{\sigma\in\Sigma(n)}S(T)_\R,$$
where $x_\sigma \in X_\Sigma$ is the $T$-fixed point corresponding to a full dimensional cone $\sigma$. Let $q \in S(H)^W_\R$ be a $W$-invariant polynomial. We recall from \cite[Section 3]{Kaveh-Manon} that $q$ can be extended to a well-defined function $\tilde{q}:\tilde{\mathfrak{B}}(G)\to\R$. The following is from \cite[Theorem 3.4]{Kaveh-Manon}.
\begin{theorem} \label{th-char-class}
    Let $G$ be a reductive algebraic group over $\mathbf{k}$. Let $\mathcal{P}$ be a framed toric principal $G$-bundle on a complete toric variety $X_\Sigma$ with the corresponding piecewise linear map $\Phi:\abs{\Sigma}\to\tilde{\mathfrak{B}}(G)$. Let $q\in S(H)_\R^W$
    be a $W$-invariant polynomial. Then the image of $q$ under the equivariant Chern-Weil homomorphism is given by the piecewise polynomial function $\tilde{q}\circ\Phi$.
\end{theorem}

\section{Background on moduli spaces}
In this short section we briefly recall the definitions of fine and coarse moduli spaces. Let $\Sch$ denote the category of schemes (of finite type over $\Spec(\Z)$).
\begin{definition}
    Let $\mathcal{A}$ be a class of objects in the category $\Sch$. Then for any $B\in\Sch$, a \emph{family} of $\mathcal{A}$-objects over $B$ is an object $X\in\Sch$ with a surjective morphism $\pi:X\to B$ such that $\pi^{-1}(b)\in\mathcal{A}$, $\forall b\in B$.
\end{definition}

\begin{definition}
    A \emph{moduli problem} for a class $\mathcal{A}$ of objects in the category $\Sch$ is a contravariant functor:
    $$\mathcal{F}_\mathcal{A}:\Sch\to \textbf{Sets}$$ that associates to each $B\in\Sch$, the set of isomorphism classes of families of $\mathcal{A}$-objects over $B$.
    To each morphism $f:B'\to B$ is associated the set map sending a family $X\to B$ to the pullback family $f^*(X)\to B'$.
\end{definition}

For any object $\mathcal{M}\in\Sch$, there is a contravariant functor
\begin{equation*}
    \Hom_\Sch(-,\mathcal{M}):\Sch\to\textbf{Sets},\ \ X\mapsto\Hom_\Sch(X,\mathcal{M}).
\end{equation*}

\begin{definition}
    An object $\mathcal{M}\in\Sch$ is called a \emph{fine moduli space} for the moduli problem $\mathcal{F}_\mathcal{A}$ if $\Hom_\Sch(-,\mathcal{M})\simeq\mathcal{F}_\mathcal{A}$.
\end{definition}

\begin{definition}
    Let $\mathcal{F}_\mathcal{A}$ be a moduli functor for a class $\mathcal{A}$ of objects in the category $\Sch$. An object $\mathcal{M}\in\Sch$ together with a natural transformation $\eta:\mathcal{F}_\mathcal{A}\to\Hom_\Sch(-,\mathcal{M})$ is a
    \emph{coarse moduli space} for $\mathcal{F}_\mathcal{A}$ if:
    \begin{enumerate}
        \item $\eta$ evaluated at a point is a set bijection
        \item For any other object $\mathcal{M}'\in\Sch$ and natural transformation $\eta':\mathcal{F}_\mathcal{A}\to\Hom_\Sch(-,\mathcal{M}')$, there is a unique morphism from $\mathcal{M}$ to $\mathcal{M}'$ such that the associated natural transformation diagram commutes.
    \end{enumerate}
\end{definition}

\section{Construction of the moduli of framed toric principal bundles}
\label{sec-moduli-tpbs}
With notation as before, let $G$ be a reductive algebraic group over $\k$ and let $X_\Sigma$ be a complete toric variety. In this section, we give a construction of a coarse moduli space of framed toric principal $G$-bundles over $X_\Sigma$ with given total equivariant characteristic class. We believe that this moduli space is indeed a fine moduli space (Section \ref{sec-fine-moduli}).

For the remainder of the paper, we fix a maximal torus $H \subset G$ and a Borel subgroup $B \subset G$ containing $H$.

Let $(\mathcal{P},p_0)$ be a framed toric principal $G$-bundle over $X_\Sigma$ with corresponding integral piecewise linear map $\Phi:\abs{\Sigma}\to\tilde{\mathfrak{B}}(G)$. For any $\rho\in\Sigma(1)$, let $Q_\rho$ denote the standard parabolic subgroup that is conjugate to the parabolic subgroup $P_\rho$, the parabolic subgroup
associated to the equivalence class of one-parameter subgroups $\Phi(v_\rho)$.
Since the normalizer of a parabolic subgroup is itself, there exists a unique coset $x_\rho Q_\rho\in G/Q_\rho$ such that $P_\rho = x_\rho Q_\rho x_\rho^{-1}$.

\begin{lemma}\label{lem-cont-tori}
    Let $\lambda:\mathbb{G}_m\to G$ be a one-parameter subgroup and $H\subset G$ a maximal torus. Then there exists a $\lambda'\sim\lambda$ with $\lambda'\in\Lambda^\vee(H)$ if and only if $P_\lambda$ contains $H$.
\end{lemma}
\begin{proof}
    Suppose there exists a $\lambda'\sim\lambda$ with $\lambda'\in\Lambda^\vee(H)$. For any $h\in H$,
    \begin{equation*}
        \lim_{s\to0}\lambda'(s)h\lambda'(s)^{-1}=\lim_{s\to0}\lambda'(s)\lambda'(s)^{-1}h=h\in G,
    \end{equation*}
    which shows that $h\in P_{\lambda'}$. Therefore, $H\subset P_{\lambda'}=P_{\lambda}$.
Conversely, suppose $H\subset P_\lambda$. We know there exists a maximal torus $\tilde{H}\subset P_\lambda$ such that $\lambda\in\Lambda^\vee(\tilde{H})$. Since all maximal tori in $P_\lambda$ are conjugate, there exists a $g\in P_\lambda$ such that $g\tilde{H}g^{-1}=H$. Let $\lambda^\prime=g\lambda g^{-1}$. We have $\lambda^\prime\in\Lambda^\vee(H)$. Since
    \begin{equation*}
        \lim_{s\to0}\lambda'(s)\lambda(s)^{-1}=\lim_{s\to0}g\lambda(s) g^{-1}\lambda(s)^{-1}=g\lim_{s\to0}\lambda(s) g^{-1}\lambda(s)^{-1},
    \end{equation*} 
    and $g^{-1}\in P_\lambda$, we know $\displaystyle\lim_{s\to0}\lambda'(s)\lambda(s)^{-1}$ exists in $G$, i.e. $\lambda'\sim\lambda$.
\end{proof}

\begin{lemma}\label{lem-fixed-pt}
    A maximal torus $H$ is contained in $x_\rho Q_\rho x_\rho^{-1}$ if and only if $x_\rho Q_\rho\in G/Q_\rho$ is an $H$-fixed point for the left action of $H$ on $G/Q_\rho$.
\end{lemma}
\begin{proof}
    \begin{equation*}
        Hx_\rho Q_\rho = x_\rho Q_\rho \Longleftrightarrow x_\rho^{-1} H x_\rho\subset Q_\rho \Longleftrightarrow H\subset x_\rho Q_\rho x_\rho^{-1}.
    \end{equation*}
\end{proof}

Recall that $x_\rho Q_\rho x_\rho^{-1} = P_\rho$, the parabolic subgroup associated to $\Phi(v_\rho)$.
\begin{prop} The tuple
    $(x_\rho Q_\rho\mid\rho\in\Sigma(1))\in\prod_{\rho\in\Sigma(1)}G/Q_\rho$ satisfies the following condition: $\forall\sigma\in\Sigma$, there exists a maximal torus $H_\sigma$ such that $\forall\rho\in\sigma(1)$, $x_\rho Q_\rho$ is an $H_\sigma$-fixed point in $G/Q_\rho$.
\end{prop}
\begin{proof}
    From Definition \ref{def-pwl}, for any $\sigma\in\Sigma$, there exists a maximal torus $H_\sigma\subset G$ such that $\forall\rho\in\sigma(1)$, $\Phi(v_\rho)\in\Lambda^\vee(H_\sigma)$.
    By Lemma \ref{lem-cont-tori}, this means that  $\forall\rho\in\sigma(1)$, $H_\sigma\subset P_\rho=x_\rho Q_\rho x_\rho^{-1}$. By Lemma \ref{lem-fixed-pt}, this is equivalent to $\forall\rho\in\sigma(1)$, $x_\rho Q_\rho$ being an $H_\sigma$-fixed point in $G/Q_\rho$.
\end{proof}

 Now let $\Psi: S^W_\R(H) \to \PP^*(\Sigma)$ be an $\R$-algebra homomorphism where $\PP^*(\Sigma)$ is the algebra of piecewise polynomial functions on $N_\R$ with respect to $\Sigma$. Consider framed toric principal bundles $\mathcal{P}$ with fixed total equivariant characteristic class $\Psi$. Recall that, by Theorem \ref{th-char-class}, a toric principal bundle $\mathcal{P}$ with piecewise linear map $\Phi$ has total equivariant characteristic class $\Psi$ if for any invariant polynomial $q\in S(H)_\R^W$, we have $\tilde{q}\circ\Phi = \Psi(q)$. 

The choice of $\Psi$ naturally gives us a collection of standard parabolic subgroups $\{Q_\rho \mid \rho \in \Sigma(1)\}$ as follows. For each ray $\rho$ we get a function $S^W_\R(H) \to \R$ by $q \mapsto \Psi(q)(v_\rho)$. This function exactly corresponds to a $W$-orbit in $\Lambda^\vee_\R(H)$. We let $Q_\rho$ be the parabolic subgroup corresponding to this $W$-orbit.

\begin{definition}\label{def-comp}
    Fix an $\R$-algebra homomorphism $\Psi: S^W_\R(H) \to \PP^*(\Sigma)$. Let $\{Q_\rho \mid \rho\in\Sigma(1)\}$ be the collection of standard parabolic subgroups determined by the choice of $\Psi$. We say that a point $(x_\rho Q_\rho\mid\rho\in\Sigma(1))\in\prod_{\rho\in\Sigma(1)}G/Q_\rho$ satisfies the \emph{compatibility conditions} with respect to $\Psi$, if $\forall\sigma\in\Sigma$ there exists a maximal torus $H_\sigma$ of $G$ that is contained in $\bigcap_{\rho\in\sigma(1)}x_\rho Q_\rho x_\rho^{-1}$ and there exists a point $(\gamma_\rho \mid \rho\in\sigma(1)) \in\prod_{\rho\in\sigma(1)}\Lambda^\vee(H_\sigma)$ satisfying the following conditions:
    \begin{enumerate}
        \item $\forall\rho\in\sigma(1)$, $P_{\gamma_\rho} = x_\rho Q_\rho x_\rho^{-1}$;
        \item $\forall \rho\in\sigma(1)$, $\forall q\in S(H_\sigma)_\R^{W_\sigma}$, $q(\gamma_\rho)=\Psi(q)(v_\rho)$, where $W_\sigma$ is the Weyl group of $(G, H_\sigma)$;
        \item $\forall \rho\in\sigma(1)$, $\gamma_\rho$ satisfies the same linear relations as the primitive ray generators $v_\rho$.
    \end{enumerate}
\end{definition}

\begin{remark}
    When $\sigma$ is a simplicial cone, the vectors $v_\rho$ are linearly independent. In this case, the condition (3) in Definition \ref{def-comp} is satisfied automatically.
\end{remark}

\begin{theorem}  \label{th-moduli-comp}
The framed toric principal bundles on $X_\Sigma$ with given total equivariant characteristic class $\Psi$ are in one-to-one correspondence with the points in $\prod_{\rho \in \Sigma(1)} G/Q_\rho$ satisfying the compatibility conditions in Definition \ref{def-comp}.    
\end{theorem}
\begin{proof}
This follows from Theorem \ref{th-KM}.    
\end{proof}

\begin{definition}
In light of Theorem \ref{th-moduli-comp}, we denote the collection of $\left(x_\rho Q_\rho\mid\rho\in\Sigma(1)\right)\in\prod_{\rho\in\Sigma(1)}G/ Q_\rho$ satisfying the compatibility conditions with respect to $\Psi$ in Definition \ref{def-comp} by $\mathcal{M}_{G,\Psi}^{\fr}$.
\end{definition}

Next we show that $\mathcal{M}_{G,\Psi}^{\fr}$ is a locally closed subset of $\prod_{\rho \in \Sigma(1)} G/Q_\rho$. We need the following lemma.
\begin{lemma}\label{lem-orbit-map}
Take $\sigma \in \Sigma$. 
    There exists a maximal torus $H_\sigma=gHg^{-1}\subset G$ such that $\displaystyle H_\sigma\subset\bigcap_{\rho\in\sigma(1)}x_\rho Q_\rho x_\rho^{-1}$ if and only if $(x_\rho Q_\rho\mid\rho\in\sigma(1))$ is in the image of the orbit map: $$G\times\prod_{\rho\in\sigma(1)}(G/Q_\rho)^{H}\to\prod_{\rho\in\sigma(1)}G/Q_\rho.$$
\end{lemma}
\begin{proof}
One knows that for any $G$-space $X$ we have $g\cdot X^H = X^{gHg^{-1}}$. From Lemma \ref{lem-fixed-pt} we then have:
    \begin{equation*}
        \begin{split}
            &\exists H_\sigma=gHg^{-1}\text{ s.t. } H_\sigma\subset\bigcap_{\rho\in\sigma(1)}x_\rho Q_\rho x_\rho^{-1}\\
            \Longleftrightarrow&\exists H_\sigma=gHg^{-1}\text{ s.t. } (x_\rho Q_\rho\mid\rho\in\sigma(1))\in\prod_{\rho\in\sigma(1)}(G/Q_\rho)^{H_\sigma}\\
            \Longleftrightarrow&\exists g\in G\text{ s.t. } (x_\rho Q_\rho\mid\rho\in\sigma(1))\in\prod_{\rho\in\sigma(1)}g\cdot(G/Q_\rho)^{H}\\
            \Longleftrightarrow&(x_\rho Q_\rho\mid\rho\in\sigma(1)) = g\cdot x,\ x\in\prod_{\rho\in\sigma(1)}(G/Q_\rho)^H.
        \end{split}
    \end{equation*}
\end{proof}

\begin{theorem}
     $\mathcal{M}_{G,\Psi}^{\fr}$ is a locally closed subvariety of $\prod_{\rho\in\Sigma(1)}G/ Q_\rho$.
\end{theorem}

\begin{proof}
    For any $\sigma\in\Sigma$, let $\pi_\sigma$ be the projection map from $\prod_{\rho\in\Sigma(1)}G/ Q_\rho$ to $\prod_{\rho\in\sigma(1)}G/ Q_\rho$.
    We have \[\pi_\sigma(\mathcal{M}_{G,\Psi}^{\fr}) \subset G\cdot (\prod_{\rho\in\sigma(1)}(G/ Q_\rho)^H).\]
    Note that $\pi_\sigma(\mathcal{M}_{G,\Psi}^{\fr})$ is $G$-invariant and hence is a union of $G$-orbits.
    The condition (1) in Definition \ref{def-comp}, specifies which face of a Weyl chamber $\gamma_\rho$ lies on. The condition (2) specifies which Weyl group orbit $\gamma_\rho$ lies on, and the condition (3) imposes possible further restrictions. Since $(G/ Q_\rho)^H$ is finite, the conditions (1), (2) and (3) state that for a certain subset (possibly empty) $S_\sigma \subset \prod_{\rho \in \sigma(1)} (G/ Q_\rho)^H$ we have \[\pi_\sigma(\mathcal{M}_{G,\Psi}^{\fr})= G\cdot S_\sigma.\]
    Therefore,
    \begin{equation}
        \mathcal{M}_{G,\Psi}^{\fr}=\bigcap_{\sigma\in\Sigma(n)}\pi_\sigma^{-1} (G\cdot S_\sigma).
    \end{equation}
    We know $G$-orbit of a finite set is locally closed. Since finite intersections and the pre-image under a continuous map of locally closed sets are locally closed, we have $\mathcal{M}_{G,\Psi}^{\fr}$ is a locally closed subvariety of $\prod_{\rho\in\Sigma(1)}G/ Q_\rho$.
\end{proof}


\section{Families of toric principal bundles}   \label{sec-fine-moduli}
In this section we sketch ideas on how to show that the moduli space $\mathbb{M}^\fr_{G, \Psi}$ is a \emph{fine} moduli space. Let $S$ be a scheme over $\k$, and let $T_S$ be the relative torus $T\times S$.
For schemes $X$ and $Y$ over $\k$, by $X \times Y$ we mean $X \times_{\Spec(\k)} Y$.

\begin{definition}
An $S$-family of toric principal $G$-bundles on $X_\Sigma$ is a toric principal $G$-bundle $\mathcal{E}_S$ on $X_\Sigma\times S$ with an action of $T_S$ compatible with the action on $X_\Sigma\times S$. An $S$-family of framed toric principal $G$-bundles is an $S$-family $\mathcal{E}_S$ of toric principal $G$-bundles together with a choice of an $S$-point over the $S$-point $x_0 \times S$. 
\end{definition}

We expect the following to be true:
Let $S$ be an affine scheme over $\k$. Let $U_\sigma$ be the affine toric variety corresponding to a cone $\sigma \subset N_\R$. Let $\mathcal{P}_{\sigma, S}$ be a toric principal $G$-bundle over the scheme $U_\sigma \times S$. Then there is a family of homomorphisms $\phi_{\sigma, S}: T_S:=T \times S \to G$ over $S$ such that $\mathcal{P}_{\sigma, S}$ is equivariantly isomorphic over $S$ to $G \times U_{\sigma} \times S$ (here $T_S$ acts on $G \times U_\sigma \times S$ by acting on $U_\sigma \times S$ in the usual manner and acting on $G$ via the family of homomorphisms $\phi_{\sigma, S}$). That is, we have the following commutative diagram:
$$
\begin{tikzcd}
\mathcal{P}_{\sigma, S} \arrow[r, "\cong"] \arrow[d] & G \times U_\sigma \times S \arrow[dl] \\ U_\sigma \times S
\end{tikzcd}
$$

Using the above, one should be able to give a classification of $S$-families of framed toric principal bundles extending \cite{Kaveh-Manon} classification (Theorem \ref{th-KM}). This then will immediately imply the following:  
Let $\mathbb{M}_{G, \Psi}^{\fr}:\Sch\to\textbf{Sets}$ be the moduli functor
\begin{equation*}
    \mathbb{M}_{G, \Psi}^{\fr}(S)=\left\{\begin{aligned}&\text{isomorphism classes of $S$-families of framed toric principal} \\ &\text{$G$-bundles on $X_\Sigma$ with total equivariant characteristic class $\Psi$}\end{aligned}\right\}.
\end{equation*}

Then $\mathcal{M}_{G,\Psi}^{\fr}$ is a fine moduli space via $\Hom(-,\mathcal{M}_{G,\Psi}^{\fr})\simeq\mathbb{M}_{G, \Psi}^{\fr}$.

We note that $G$ naturally acts on $\mathcal{M}^{\fr}_{G, \Psi}$ by acting on the frames.
The moduli $\mathcal{M}_{G, \Psi}$ of toric principal $G$-bundles over $X_\Sigma$ with total equivariant characteristic class $\Psi$ is the quotient stack of $\mathcal{M}^{\fr}_{G, \Psi}$ by $G$.    

\end{document}